\documentclass[11pt]{article}
\usepackage{latexsym}
\usepackage{amsfonts}
\usepackage{mathrsfs,amsthm,amsmath}
\newtheorem{proposition}{Proposition}

\newtheorem{corollary}{Corollary}

\newtheorem{remark}{Remark}

\begin{document}
\title{Asymptotic results for a multivariate version of the alternative fractional Poisson
process\thanks{The authors acknowledge the support of Gruppo
Nazionale per l'Analisi Matematica, la Probabilit\`{a} e le loro
Applicazioni (GNAMPA) of the Istituto Nazionale di Alta Matematica
(INdAM).}}
\author{Luisa Beghin\thanks{Dipartimento di Scienze Statistiche,
Sapienza Universit\`a di Roma, Piazzale Aldo Moro 5, I-00185 Roma,
Italy. e-mail: \texttt{luisa.beghin@uniroma1.it}} \and Claudio
Macci\thanks{Dipartimento di Matematica, Universit\`a di Roma Tor
Vergata, Via della Ricerca Scientifica, I-00133 Rome, Italy.
e-mail: \texttt{macci@mat.uniroma2.it}}}
\date{}
\maketitle
\begin{abstract}
\noindent A multivariate fractional Poisson process was recently
defined in \cite{BeghinMacciAAP2016} by considering a common
independent random time change for a finite dimensional vector of
independent (non-fractional) Poisson processes; moreover it was
proved that, for each fixed $t\geq 0$, it has a suitable
multinomial conditional distribution of the components given their
sum. In this paper we consider another multivariate process
$\{\underline{M}^\nu(t)=(M_1^\nu(t),\ldots,M_m^\nu(t)):t\geq 0\}$
with the same conditional distributions of the components given
their sums, and different marginal distributions of the sums; more
precisely we assume that the one-dimensional marginal
distributions of the process $\left\{\sum_{i=1}^mM_i^\nu(t):t\geq
0\right\}$ coincide with the ones of the alternative fractional
(univariate) Poisson process in \cite{BeghinMacciSPL2013}. We
present large deviation results for
$\{\underline{M}^\nu(t)=(M_1^\nu(t),\ldots,M_m^\nu(t)):t\geq 0\}$,
and this generalizes the result in \cite{BeghinMacciSPL2013}
concerning the univariate case. We also study moderate deviations
and we present some statistical applications concerning the
estimation of the fractional parameter $\nu$.\\

\end{abstract}

\section{Introduction}
Fractional Poisson processes are widely studied in the literature
by considering a version of some known equations for the
probability mass functions with fractional derivatives and/or
fractional difference operators (see \cite{Laskin},
\cite{MainardiGorenfloScalas}, \cite{BeghinOrsingherEJP2009},
\cite{BeghinOrsingherEJP2010}, \cite{OrsingherPolitoSPL2012},
\cite{PolitiKaizojiScalas} and \cite{RepinSaichev}). Typically
these processes are often represented in terms of randomly
time-changed and subordinated processes (see e.g.
\cite{KumarNaneVellaisamy} and \cite{MeerschaertNaneVellaisamy})
and appear in several applications (see e.g.
\cite{BiardSaussereau2014}, where the surplus process of an
insurance company is modeled by a compound fractional Poisson
process).

A multivariate (space and/or time) fractional Poisson process was
recently defined in \cite{BeghinMacciAAP2016} by considering a
common independent random time change in terms of the stable
subordinator and/or its inverse for a finite dimensional vector of
independent (non-fractional) Poisson processes. In the proof of
Proposition 4 in \cite{BeghinMacciAAP2016} it was proved that, for
each fixed $t\geq 0$, the conditional (joint) distribution of the
components of this multivariate process given their sum is
multinomial; moreover this conditional multinomial distribution
does not depend on $t$ and on the fractional parameters.

In this paper we consider another multivariate process
$\{\underline{M}^\nu(t)=(M_1^\nu(t),\ldots,M_m^\nu(t)):t\geq 0\}$
with the same conditional distributions of the components given
their sums, but we change the distribution of the sums of the
components. More precisely we assume that the one-dimensional
marginal distributions of the process
$\left\{\sum_{i=1}^mM_i^\nu(t):t\geq 0\right\}$ coincide with the
ones of the alternative fractional (univariate) Poisson process in
\cite{BeghinMacciSPL2013}; in other words we mean the alternative
fractional Poisson processes in \cite{BeghinOrsingherEJP2009} with
a deterministic time-change. Thus it is natural to define the
process in this paper as the multivariate version of the
alternative fractional Poisson process.

The alternative fractional Poisson process in
\cite{BeghinOrsingherEJP2009} appears as the process which counts
the number of changes of direction of a fractional telegraph
process (see e.g. (4.7) in \cite{GarraOrsingherPolitoJSP2014}),
and of a reflected random flight on the surface of a sphere (see
e.g. (4.24) in \cite{DegregorioOrsingher}). Some generalizations
of the alternative fractional Poisson process in
\cite{BeghinOrsingherEJP2009} can be found in
\cite{GarraOrsingherPolitoJAP2015} (see (3.5)) and in
\cite{PoganyTomovski} (see Proposition 2.1). In all these cases we
have a weighted Poisson process as in
\cite{BalakrishnanKozubowski}; the concept of weighted Poisson
process for $\left\{\sum_{i=1}^mM_i^\nu(t):t\geq 0\right\}$ is
illustrated in Remark \ref{rem:wpp}.

The aim of this paper is to present large deviation results for
the multivariate version of the alternative fractional Poisson
process. The theory of large deviations gives an asymptotic
computation of small probabilities on exponential scale (see e.g.
\cite{DemboZeitouni} as references on this topic). The main
results in this paper are Propositions \ref{prop:LD} and
\ref{prop:MD}, which concern large and moderate deviations. The
main tool used in the proofs of Propositions \ref{prop:LD} and
\ref{prop:MD} is the G\"{a}rtner Ellis Theorem (see e.g. Section
2.3 in \cite{DemboZeitouni}). We point out that in
\cite{BeghinMacciSPL2013} we study large deviations only; in
particular Proposition \ref{prop:LD} in this paper reduces to
Proposition 4.1 in \cite{BeghinMacciSPL2013} if we consider the
univariate case $m=1$ (see Remark \ref{rem:LD-m=1}).

The term moderate deviations is used for a class of large
deviation principles governed by the same quadratic rate function
which uniquely vanishes at the origin. Typically moderate
deviations fill the gap between a convergence to zero and an
asymptotic Normality result. We also recall that, as pointed out
in some references (see e.g. \cite{BrycSPL1993} and the references
cited therein), under certain conditions one can obtain the weak
convergence to a centered Normal distribution whose variance is
determined by a large deviation principle obtained by the
G\"{a}rtner Ellis Theorem.

We conclude with the outline of the paper. We start with some
preliminaries in Section \ref{sec:preliminaries}. The multivariate
process studied in this paper is defined in Section
\ref{sec:def-alt-mult-FPP}. Large and moderate deviation results
are presented in Section \ref{sec:LDMDresults}. We conclude with
some statistical applications in Section
\ref{sec:statistical-applications}.

\section{Preliminaries}\label{sec:preliminaries}
We always set $0\log 0=0$. In general we deal with vectors in
$\mathbb{R}^m$ and we use the following notation:
$\underline{x}=(x_1,\ldots,x_m)$, and $\underline{0}=(0,\ldots,0)$
is the null vector; $\underline{x}\geq\underline{0}$ means that
$x_1,\ldots,x_m\geq 0$; we set $s(\underline{x})=\sum_{i=1}^mx_i$
and
$\langle\underline{x},\underline{y}\rangle=\sum_{i=1}^mx_iy_i$.

\subsection{Preliminaries on large (and moderate) deviations}
We recall the basic definitions (see e.g. \cite{DemboZeitouni},
pages 4-5). Let $\mathcal{Z}$ be a Hausdorff topological space
with Borel $\sigma$-algebra $\mathcal{B}_{\mathcal{Z}}$. A speed
function is a family of numbers $\{v_t:t>0\}$ such that
$\lim_{t\to\infty}v_t=\infty$. A lower semi-continuous function
$I:\mathcal{Z}\to [0,\infty]$ is called rate function. A family of
$\mathcal{Z}$-valued random variables $\{Z_t:t>0\}$ satisfies the
\emph{large deviation principle} (LDP for short), as $t\to\infty$,
with speed function $v_t$ and rate function $I$ if
$$\limsup_{t\to\infty}\frac{1}{v_t}\log P(Z_t\in F)\leq-\inf_{z\in F}I(z)\ (\mbox{for all closed sets}\ F)$$
and
$$\liminf_{t\to\infty}\frac{1}{v_t}\log P(Z_t\in G)\geq-\inf_{z\in G}I(z)\ (\mbox{for all open sets}\ G).$$
A rate function $I$ is said to be good if all the level sets
$\{\{z\in\mathcal{Z}:I(z)\leq\gamma\}:\gamma\geq 0\}$ are compact.

The term \emph{moderate deviations} is used when, for all positive
numbers $\{a_t:t>0\}$ such that
\begin{equation}\label{eq:MD-conditions}
a_t\to 0\ \mbox{and}\ ta_t\to\infty\ (\mbox{as}\ t\to\infty),
\end{equation}
we have a LDP for suitable centered random variables on
$\mathcal{Z}=\mathbb{R}^m$ (for some $m\geq 1$) with speed $1/a_t$
and the same quadratic rate function which uniquely vanishes at
the origin of $\mathbb{R}^m$ (we mean that the rate function does
not depend on the choice of $\{a_t:t>0\}$). Typically moderate
deviations fill the gap between two regimes (for the second one
see Remark \ref{rem:Asymptotic-Normality}):
\begin{itemize}
\item a convergence (at least in probability) to zero of centered random
variables (case $a_t=\frac{1}{t}$);
\item a weak convergence to a centered Normal distribution (case
$a_t=1$).
\end{itemize}
Note that in both case one condition in \eqref{eq:MD-conditions}
fails.

\subsection{Preliminaries on (generalized) Mittag-Leffler functions}
Let
\begin{equation}\label{eq:def-MLfunction}
E_{\alpha,\beta}(x):=\sum_{r\geq 0}\frac{x^r}{\Gamma(\alpha
r+\beta)}
\end{equation}
be the Mittag-Leffler function (see e.g. \cite{Podlubny}, page
17), and let
$$E_{\alpha,\beta}^\gamma(x):=\sum_{j\geq 0}\frac{(\gamma)^{(j)}x^j}{j!\Gamma(\alpha j+\beta)}$$
be the generalized Mittag-Leffler function (see e.g. (1.9.1) in
\cite{KilbasSrivastavaTrujillo}) where
$$(\gamma)^{(j)}:=\left\{\begin{array}{ll}
\gamma(\gamma+1)\cdots (\gamma+j-1)&\ \mathrm{if}\ j\geq 1\\
1&\ \mathrm{if}\ j=0,
\end{array}\right.$$
is the rising factorial, also called Pochhammer symbol (see e.g.
(1.5.5) in \cite{KilbasSrivastavaTrujillo}). Note that we have
$E_{\alpha,\beta}^1$, i.e. $E_{\alpha,\beta}^\gamma$ with
$\gamma=1$, coincides with $E_{\alpha,\beta}$ in
\eqref{eq:def-MLfunction}. In view of what follows (see e.g.
(1.8.27) in \cite{KilbasSrivastavaTrujillo}) we recall that, if we
use the symbol $\sim$ to mean that the ratio tends to 1, we have
\begin{equation}\label{eq:ML-asymptotic}
E_{\nu,\beta}(z)\sim\frac{1}{\nu}z^{(1-\beta)/\nu}e^{z^{1/\nu}}\
\mbox{as}\ z\to\infty;
\end{equation}
actually we can say that
\begin{equation}\label{eq:ML-asymptotic-with-remainder}
E_{\nu,\beta}(z)=\frac{1}{\nu}z^{(1-\beta)/\nu}e^{z^{1/\nu}}+r(z),\
\mbox{where}\ r(z)\to 0\ \mbox{as}\ z\to\infty.
\end{equation}

\section{An alternative multivariate fractional Poisson process}\label{sec:def-alt-mult-FPP}
Let $\underline{\lambda}\in(0,\infty)^m$ be arbitrarily fixed
(actually we could consider
$\underline{\lambda}\in[0,\infty)^m\setminus\{\underline{0}\}$
with suitable modifications). We present a multivariate fractional
Poisson process $\{\underline{M}^\nu(t):t\geq 0\}$ where, as in
the proof of Proposition 4 in \cite{BeghinMacciAAP2016}, for all
$t\geq 0$ we consider the following conditional multinomial
distribution of $(M_1^\nu(t),\ldots,M_m^\nu(t))$ given their sum
$s(\underline{M}^\nu(t))=\sum_{i=1}^mM_i^\nu(t)$:
$$P(\underline{M}^\nu(t)=\underline{k}|s(\underline{M}^\nu(t))=s(\underline{k}))=\frac{(s(\underline{k}))!}{k_1!\cdots k_m!}
\prod_{i=1}^m\left(\frac{\lambda_i}{s(\underline{\lambda})}\right)^{k_i}\
\mbox{for all integers}\ k_1,\ldots,k_m\geq 0.$$ Moreover we
assume that the one-dimensional marginal distributions of the sum
process $\{s(\underline{M}^\nu(t)):t\geq 0\}$ coincide with the
ones of the alternative fractional (univariate) Poisson process in
\cite{BeghinMacciSPL2013} with parameter $s(\underline{\lambda})$
(in place of $\lambda$), i.e.
$$P(s(\underline{M}^\nu(t))=h)=\frac{(s(\underline{\lambda})t^\nu)^h}{\Gamma(\nu h+1)}\cdot\frac{1}{E_{\nu,1}(s(\underline{\lambda})t^\nu)}
\ \mbox{for all integer}\ h\geq 0.$$

\begin{remark}[Weighted Poisson process]\label{rem:wpp}
For all $t\geq 0$ we have
$$P(s(\underline{M}^\nu(t))=h)=\frac{w(h)\frac{(s(\underline{\lambda})t^\nu)^h}{h!}e^{-s(\underline{\lambda})t^\nu}}
{\sum_{j\geq
0}w(j)\frac{(s(\underline{\lambda})t^\nu)^j}{j!}e^{-s(\underline{\lambda})t^\nu}}\
\mbox{for all integer}\ h\geq 0,$$ where
$w(h):=\frac{h!}{\Gamma(\nu h+1)}$.
\end{remark}

Thus, for each fixed $t\geq 0$, we consider the following
multivariate probability mass function for the random variable
$\underline{M}^\nu(t)$:
\begin{multline*}
P(\underline{M}^\nu(t)=\underline{k})=P(\underline{M}^\nu(t)=\underline{k}|s(\underline{M}^\nu(t))=s(\underline{k}))P(s(\underline{M}^\nu(t))=s(\underline{k}))\\
=\frac{(s(\underline{k}))!}{k_1!\cdots k_m!}
\prod_{i=1}^m\left(\frac{\lambda_i}{s(\underline{\lambda})}\right)^{k_i}\cdot
\frac{(s(\underline{\lambda})t^\nu)^{s(\underline{k})}}{\Gamma(\nu(s(\underline{k}))+1)}\cdot\frac{1}{E_{\nu,1}(s(\underline{\lambda})t^\nu)}\\
=\frac{(s(\underline{k}))!}{k_1!\cdots k_m!}
\prod_{i=1}^m\lambda_i^{k_i}\cdot\frac{(t^\nu)^{s(\underline{k})}}{\Gamma(\nu(s(\underline{k}))+1)}\cdot\frac{1}{E_{\nu,1}(s(\underline{\lambda})t^\nu)}\
\mbox{for all integers}\ k_1,\ldots,k_m\geq 0.
\end{multline*}
The moment generating functions of the ($m$-variate) random
variables $\{\underline{M}^\nu(t):t\geq 0\}$, with argument
$\underline{\theta}\in\mathbb{R}^m$, are
\begin{multline*}
\mathbb{E}\left[e^{\langle\underline{\theta},\underline{M}^\nu(t)\rangle}\right]
=\sum_{\underline{k}\geq\underline{0}}e^{\sum_{i=1}^m\theta_ik_i}P(\underline{M}^\nu(t)=\underline{k})\\
=\frac{1}{E_{\nu,1}(s(\underline{\lambda})t^\nu)}\sum_{\underline{k}\geq\underline{0}}\frac{(s(\underline{k}))!}{k_1!\cdots
k_m!}\prod_{i=1}^m(e^{\theta_i}\lambda_i)^{k_i}\cdot\frac{(t^\nu)^{s(\underline{k})}}{\Gamma(\nu(s(\underline{k}))+1)}\\
=\frac{1}{E_{\nu,1}(s(\underline{\lambda})t^\nu)}\sum_{r\geq
0}\frac{\left(\left(\sum_{i=1}^m\lambda_ie^{\theta_i}\right)t^\nu\right)^r}{\Gamma(\nu
r+1)}
\end{multline*}
and therefore
\begin{equation}\label{eq:MGF}
\mathbb{E}\left[e^{\langle\underline{\theta},\underline{M}^\nu(t)\rangle}\right]
=\frac{E_{\nu,1}\left(\left(\sum_{i=1}^m\lambda_ie^{\theta_i}\right)t^\nu\right)}{E_{\nu,1}(s(\underline{\lambda})t^\nu)}.
\end{equation}
Moreover the expected values are
\begin{equation}\label{eq:expected-value}
\mathbb{E}[\underline{M}^\nu(t)]=\nabla\left.\mathbb{E}\left[e^{\langle\underline{\theta},\underline{M}^\nu(t)\rangle}\right]
\right|_{\underline{\theta}=\underline{0}}
=\frac{E_{\nu,\nu+1}^2(s(\underline{\lambda})t^\nu)}{E_{\nu,1}(s(\underline{\lambda})t^\nu)}t^\nu\underline{\lambda}
=\frac{E_{\nu,\nu}(s(\underline{\lambda})t^\nu)}{E_{\nu,1}(s(\underline{\lambda})t^\nu)}\frac{t^\nu}{\nu}\underline{\lambda}
\end{equation}
by (1.8.22) in \cite{KilbasSrivastavaTrujillo} and some
computations with generalized Mittag-Leffler functions; note that,
if we set $m=1$ and if we replace $t^\nu$ with $t$, formula
\eqref{eq:expected-value} meets (4.6) in
\cite{BeghinOrsingherEJP2009}.

\section{Large and moderate deviations}\label{sec:LDMDresults}
We start with large deviations.

\begin{proposition}\label{prop:LD}
The family of random variables
$\left\{\frac{\underline{M}^\nu(t)}{t}:t>0\right\}$ satisfies the
LDP with speed $v_t=t$ and good rate function $\Lambda^*$ defined
by
$$\Lambda^*(\underline{x}):=\left\{\begin{array}{ll}
\sum_{i=1}^mx_i\log\left(\frac{\nu^\nu}{\lambda_i}\frac{x_i}{(s(\underline{x}))^{1-\nu}}\right)-\nu
s(\underline{x})+(s(\underline{\lambda}))^{1/\nu}&\ \mbox{if}\ \underline{x}\in[0,\infty)^m\\
\infty&\ \mbox{otherwise}
\end{array}\right.$$
(we recall that $0\log 0=0$).
\end{proposition}
\begin{proof}
We apply G\"{a}rtner Ellis Theorem. Then, by taking into account
\eqref{eq:MGF} and \eqref{eq:ML-asymptotic}, for all
$\underline{\theta}\in\mathbb{R}^m$ we have
\begin{equation}\label{eq:GElimit-LD}
\lim_{t\to\infty}\frac{1}{t}\log\mathbb{E}\left[e^{\langle\underline{\theta},\underline{M}^\nu(t)\rangle}\right]
=\left(\sum_{i=1}^m\lambda_ie^{\theta_i}\right)^{1/\nu}-\left(\sum_{i=1}^m\lambda_i\right)^{1/\nu}
=\left(\sum_{i=1}^m\lambda_ie^{\theta_i}\right)^{1/\nu}-(s(\underline{\lambda}))^{1/\nu}=:\Lambda(\underline{\theta}).
\end{equation}
So, since $\Lambda$ is finite everywhere and differentiable, the
LDP holds with speed $v_t=t$ and good rate function $\Lambda^*$
defined by
$$\Lambda^*(\underline{x}):=\sup_{\underline{\theta}\in\mathbb{R}^m}\{\langle\underline{\theta},\underline{x}\rangle-\Lambda(\underline{\theta})\}.$$

We conclude the proof showing that this rate function coincides
with the one in the statement.\\
$\bullet$ The case $\underline{x}\notin[0,\infty)^m$ is trivial;
in fact, in this case, we have $x_i<0$ for some
$i\in\{1,\ldots,m\}$, and therefore
$\Lambda^*(\underline{x})=\infty$ by taking $\theta_j=0$ for
$j\neq i$, and by letting $\theta_i\to-\infty$. On the other hand
we have $\Lambda^*(\underline{x})=\infty$ because
$P(\underline{M}^\nu(t)/t\in [0,\infty)^m)=1$ for all $t>0$, and
$[0,\infty)^m$ is a closed set.\\
$\bullet$ For $\underline{x}\in(0,\infty)^m$ we consider the
system of equations (for $i\in\{1,\ldots,m\}$)
$$x_i=\frac{\partial}{\partial\theta_i}\Lambda(\underline{\theta}),\
\mbox{i.e.}\
x_i=\frac{1}{\nu}\left(\sum_{j=1}^m\lambda_je^{\theta_j}\right)^{1/\nu-1}\lambda_ie^{\theta_i},$$
and we have a unique solution
$\underline{\theta}(\underline{x})=(\theta_1(\underline{x}),\ldots,\theta_m(\underline{x}))$
defined by
$$\theta_i(\underline{x})=\log\left(\frac{\nu^\nu}{\lambda_i}\frac{x_i}{(s(\underline{x}))^{1-\nu}}\right);$$
in fact
\begin{multline*}
\frac{1}{\nu}\left(\sum_{j=1}^m\lambda_je^{\theta_j(\underline{x})}\right)^{1/\nu-1}\lambda_ie^{\theta_i(\underline{x})}
=\frac{1}{\nu}\left(\sum_{j=1}^m\lambda_j\cdot\frac{\nu^\nu}{\lambda_j}\frac{x_j}{(s(\underline{x}))^{1-\nu}}\right)^{1/\nu-1}
\lambda_i\cdot\frac{\nu^\nu}{\lambda_i}\frac{x_i}{(s(\underline{x}))^{1-\nu}}\\
=\frac{1}{\nu}\left(\nu^\nu(s(\underline{x}))^{1-(1-\nu)}\right)^{1/\nu-1}\cdot
\nu^\nu\frac{x_i}{(s(\underline{x}))^{1-\nu}}=x_i.
\end{multline*}
Thus
\begin{multline*}
\Lambda^*(\underline{x}):=\langle\underline{\theta}(\underline{x}),\underline{x}\rangle-\Lambda(\underline{\theta}(\underline{x}))\\
=\sum_{i=1}^mx_i\log\left(\frac{\nu^\nu}{\lambda_i}\frac{x_i}{(s(\underline{x}))^{1-\nu}}\right)
-\left(\sum_{i=1}^m\lambda_i\cdot\frac{\nu^\nu}{\lambda_i}\frac{x_i}{(s(\underline{x}))^{1-\nu}}\right)^{1/\nu}+(s(\underline{\lambda}))^{1/\nu}\\
=\sum_{i=1}^mx_i\log\left(\frac{\nu^\nu}{\lambda_i}\frac{x_i}{(s(\underline{x}))^{1-\nu}}\right)
-\left(\nu^\nu(s(\underline{x}))^\nu\right)^{1/\nu}+(s(\underline{\lambda}))^{1/\nu}\\
=\sum_{i=1}^mx_i\log\left(\frac{\nu^\nu}{\lambda_i}\frac{x_i}{(s(\underline{x}))^{1-\nu}}\right)
-\nu s(\underline{x})+(s(\underline{\lambda}))^{1/\nu}.
\end{multline*}
$\bullet$ The final case concerns
$\underline{x}\in[0,\infty)^m\setminus (0,\infty)^m$. For
$\underline{x}=\underline{0}$ we have
$$\Lambda^*(\underline{0})=\sup_{\underline{\theta}\in\mathbb{R}^m}\left\{-\left(\sum_{i=1}^m\lambda_ie^{\theta_i}\right)^{1/\nu}
+(s(\underline{\lambda}))^{1/\nu}\right\}=(s(\underline{\lambda}))^{1/\nu}$$
by letting $\theta_1,\ldots,\theta_m\to-\infty$. For
$\underline{x}\in[0,\infty)^m\setminus
((0,\infty)^m\cup\{\underline{0}\})$ we consider the set
$\mathcal{S}(\underline{x}):=\{i\in\{1,\ldots,m\}:x_i>0\}$, and we
have $\emptyset\neq\mathcal{S}(\underline{x})\neq\{1,\ldots, m\}$.
Then
$$\Lambda^*(\underline{x}):=\sup_{\underline{\theta}\in\mathbb{R}^m}\left\{\sum_{i\in\mathcal{S}(\underline{x})}\theta_ix_i
-\left(\sum_{i=1}^m\lambda_ie^{\theta_i}\right)^{1/\nu}+(s(\underline{\lambda}))^{1/\nu}\right\}$$
and, after letting $\theta_i\to-\infty$ for
$i\notin\mathcal{S}(\underline{x})$, we can consider the system of
equations
$$x_i=\frac{1}{\nu}\left(\sum_{j=1}^m\lambda_je^{\theta_j}\right)^{1/\nu-1}\lambda_ie^{\theta_i}\ (\mbox{for}\ i\in\mathcal{S}(\underline{x}))$$
and we can adapt what we said above for
$\underline{x}\in(0,\infty)^m$.
\end{proof}

We can say that $\Lambda^*(\underline{x})=0$ if and only if
$\underline{x}=\nabla\Lambda(\underline{0})$, where
\begin{equation}\label{eq:gradiente-Lambda-0}
\nabla\Lambda(\underline{0})=\frac{1}{\nu}\cdot(s(\underline{\lambda}))^{1/\nu-1}\cdot\underline{\lambda}.
\end{equation}
In particular we can check that
\begin{multline*}
\Lambda^*(\nabla\Lambda(\underline{0}))=\sum_{i=1}^m\frac{1}{\nu}\cdot(s(\underline{\lambda}))^{1/\nu-1}\lambda_i
\log\left(\frac{\nu^\nu}{\lambda_i}\frac{\frac{1}{\nu}\cdot(s(\underline{\lambda}))^{1/\nu-1}\lambda_i}
{(\frac{1}{\nu}\cdot(s(\underline{\lambda}))^{1/\nu-1+1})^{1-\nu}}\right)\\
-\nu\cdot\frac{1}{\nu}\cdot(s(\underline{\lambda}))^{1/\nu-1+1}+(s(\underline{\lambda}))^{1/\nu}=0.
\end{multline*}

The following remarks concern Proposition \ref{prop:LD}.

\begin{remark}[The case $m=1$]\label{rem:LD-m=1}
Proposition \ref{prop:LD} here reduces to Proposition 4.1 in
\cite{BeghinMacciSPL2013} when $m=1$ (actually some parts of the
proof are simplified). In particular for the rate function
$\Lambda^*$ (with $x$ in place of $\underline{x}$ and
$s(\underline{x})$, and $\lambda$ in place of
$\underline{\lambda}$ and $s(\underline{\lambda})$) we have
$$\Lambda^*(x):=\left\{\begin{array}{ll}
x\log\left(\frac{(\nu x)^\nu}{\lambda}\right)-\nu x+\lambda^{1/\nu}&\ \mbox{if}\ x\geq 0\\
\infty&\ \mbox{if}\ x<0
\end{array}\right.
=I_{\nu,\lambda}^{(A)}(x),$$ where $I_{\nu,\lambda}^{(A)}$ is the
rate function in Proposition 4.1 in \cite{BeghinMacciSPL2013}.
\end{remark}

\begin{remark}[The case $\nu=1$]\label{rem:LD-nu=1}
We have
$$\Lambda_{(\nu=1)}^*(\underline{x}):=\left\{\begin{array}{ll}
\sum_{i=1}^m\left\{x_i\log\left(\frac{x_i}{\lambda_i}\right)-x_i+\lambda_i\right\}&\ \mbox{if}\ \underline{x}\in[0,\infty)^m\\
\infty&\ \mbox{otherwise}.
\end{array}\right.$$
Thus
$\Lambda_{(\nu=1)}^*(\underline{x})=\sum_{i=1}^mI_{1,\lambda}^{(A)}(x_i)$
for all $\underline{x}\in\mathbb{R}^m$, where
$I_{\nu,\lambda}^{(A)}$ is the rate function in Proposition 4.1 in
\cite{BeghinMacciSPL2013} (as in Remark \ref{rem:LD-m=1}); this
equality agrees the well-known independence of the one-dimensional
marginal processes $\{M_1^1(t):t\geq 0\},\ldots,\{M_m^1(t):t\geq
0\}$.
\end{remark}

\begin{remark}[An alternative expression of $\Lambda^*$]\label{rem:LD-alternative-expression-rf}
If we consider the relative entropy of a probability measure
$\underline{p}=(p_1,\ldots,p_m)$ on $\{1,\ldots,m\}$ with respect
to another one $\underline{q}=(q_1,\ldots,q_m)$, i.e.
$$H(\underline{p};\underline{q}):=\sum_{i=1}^mp_i\log\left(\frac{p_i}{q_i}\right),$$
for $\underline{x}\in [0,\infty)^m$ we have
\begin{multline*}
\Lambda^*(\underline{x})=\sum_{i=1}^mx_i\log\left(\frac{x_i/s(\underline{x})}{\lambda_i/s(\underline{\lambda})}\right)
+\sum_{i=1}^mx_i\log\left(\frac{\nu^\nu}{(s(\underline{x}))^{1-\nu}}\frac{s(\underline{x})}{s(\underline{\lambda})}\right)-\nu
s(\underline{x})+(s(\underline{\lambda}))^{1/\nu}\\
=s(\underline{x})H\left(\frac{\underline{x}}{s(\underline{x})};\frac{\underline{\lambda}}{s(\underline{\lambda})}\right)+
\underbrace{s(\underline{x})\log\left(\frac{\nu^\nu(s(\underline{x}))^\nu}{s(\underline{\lambda})}\right)-\nu
s(\underline{x})+(s(\underline{\lambda}))^{1/\nu}}_{=I_{\nu,s(\underline{\lambda})}^{(A)}(s(\underline{x}))},
\end{multline*}
where $I_{\nu,s(\underline{\lambda})}^{(A)}$ concerns the notation
used for the rate function in Proposition 4.1 in
\cite{BeghinMacciSPL2013} (see Remark \ref{rem:LD-m=1}).
Obviously, for $\underline{x}=\underline{0}$, we have
$s(\underline{x})H\left(\frac{\underline{x}}{s(\underline{x})};\frac{\underline{\lambda}}{s(\underline{\lambda})}\right)=0$.
\end{remark}

The next proposition concerns moderate deviations. In view of what
follows we need to introduce the matrix
$C=(c_{jk})_{j,k\in\{1,\ldots,m\}}$ defined by
\begin{equation}\label{eq:def-covariance-matrix}
c_{jk}^{(\nu)}:=\left\{\begin{array}{ll}
\frac{1}{\nu}\left(\frac{1}{\nu}-1\right)(s(\underline{\lambda}))^{1/\nu-2}\lambda_j\lambda_k&\ \mbox{if}\ j\neq k\\
\frac{1}{\nu}\left(\frac{1}{\nu}-1\right)(s(\underline{\lambda}))^{1/\nu-2}\lambda_j^2
+\frac{1}{\nu}(s(\underline{\lambda}))^{1/\nu-1}\lambda_j&\
\mbox{if}\ j=k
\end{array}\right.
\end{equation}
and the function $\tilde{\Lambda}$ defined by
\begin{equation}\label{eq:def-Lambda-tilde}
\tilde{\Lambda}(\underline{\theta}):=\frac{1}{2}\langle\underline{\theta},C\underline{\theta}\rangle.
\end{equation}

\begin{proposition}\label{prop:MD}
For all families of positive numbers $\{a_t:t>0\}$ such that
\eqref{eq:MD-conditions} holds, the family of random variables
$\left\{\sqrt{ta_t}\cdot\frac{\underline{M}^\nu(t)-\mathbb{E}[\underline{M}^\nu(t)]}{t}:t>0\right\}$
satisfies the LDP with speed $1/a_t$ and good rate function
$\tilde{\Lambda}^*$ defined by
$$\tilde{\Lambda}^*(\underline{x}):=\sup_{\underline{\theta}\in\mathbb{R}^m}
\{\langle\underline{\theta},\underline{x}\rangle-\tilde{\Lambda}(\underline{\theta})\}.$$
\end{proposition}
\begin{proof}
We apply G\"{a}rtner Ellis Theorem and the desired LDP holds if we
prove that
$$\lim_{t\to\infty}\underbrace{\frac{1}{1/a_t}\log
\mathbb{E}\left[e^{\frac{1}{a_t}\cdot\sqrt{ta_t}\cdot\langle\underline{\theta},\frac{\underline{M}^\nu(t)-\mathbb{E}[\underline{M}^\nu(t)]}{t}\rangle}\right]}
_{=:\Lambda_t(\underline{\theta})}=\tilde{\Lambda}(\underline{\theta})\
(\mbox{for all}\ \underline{\theta}\in\mathbb{R}^m).$$ We start
with some manipulations where we take into account \eqref{eq:MGF}
and \eqref{eq:expected-value}:
\begin{multline*}
\Lambda_t(\underline{\theta})
=a_t\log\mathbb{E}\left[e^{\frac{1}{\sqrt{ta_t}}\langle\underline{\theta},\underline{M}^\nu(t)-\mathbb{E}[\underline{M}^\nu(t)]\rangle}\right]\\
=a_t\left(\log\mathbb{E}\left[e^{\frac{1}{\sqrt{ta_t}}\langle\underline{\theta},\underline{M}^\nu(t)\rangle}\right]
-\frac{1}{\sqrt{ta_t}}\langle\underline{\theta},\mathbb{E}[\underline{M}^\nu(t)]\rangle\right)\\
=a_t\left(\log\frac{E_{\nu,1}\left(\left(\sum_{i=1}^m\lambda_ie^{\theta_i/\sqrt{ta_t}}\right)t^\nu\right)}{E_{\nu,1}(s(\underline{\lambda})t^\nu)}
-\frac{1}{\sqrt{ta_t}}\frac{E_{\nu,\nu}(s(\underline{\lambda})t^\nu)}{E_{\nu,1}(s(\underline{\lambda})t^\nu)}\frac{t^\nu}{\nu}
\langle\underline{\theta},\underline{\lambda}\rangle\right).
\end{multline*}
Thus, after some computations, we get
$$\Lambda_t(\underline{\theta})=A_1(t)+A_2(t)$$
where
$$A_1(t):=a_t\left(\log\frac{E_{\nu,1}\left(\left(\sum_{i=1}^m\lambda_ie^{\theta_i/\sqrt{ta_t}}\right)t^\nu\right)}
{\frac{1}{\nu}e^{(\sum_{i=1}^m\lambda_ie^{\theta_i/\sqrt{ta_t}})^{1/\nu}\cdot
t}}-\log\frac{E_{\nu,1}(s(\underline{\lambda})t^\nu)}{\frac{1}{\nu}e^{(s(\underline{\lambda}))^{1/\nu}\cdot
t}}\right)$$ and, if we consider the function $\Lambda$ in
\eqref{eq:GElimit-LD},
\begin{multline*}
A_2(t):=ta_t\left(\frac{1}{t}\log\frac{\frac{1}{\nu}e^{(\sum_{i=1}^m\lambda_ie^{\theta_i/\sqrt{ta_t}})^{1/\nu}\cdot
t}}{\frac{1}{\nu}e^{(s(\underline{\lambda}))^{1/\nu}\cdot
t}}-\frac{1}{\sqrt{ta_t}}\frac{E_{\nu,\nu}(s(\underline{\lambda})t^\nu)}{E_{\nu,1}(s(\underline{\lambda})t^\nu)}\frac{t^{\nu-1}}{\nu}
\langle\underline{\theta},\underline{\lambda}\rangle\right)\\
=ta_t\left(\Lambda\left(\frac{1}{\sqrt{ta_t}}\cdot\underline{\theta}\right)
-\frac{1}{\sqrt{ta_t}}\frac{E_{\nu,\nu}(s(\underline{\lambda})t^\nu)}{E_{\nu,1}(s(\underline{\lambda})t^\nu)}
\frac{t^{\nu-1}}{\nu}\langle\underline{\theta},\underline{\lambda}\rangle\right).
\end{multline*}

Then, for all $\underline{\theta}\in\mathbb{R}^m$, we have
$A_1(t)\to 0$ as $t\to\infty$ (this is a consequence of $a_t\to
0$, stated in \eqref{eq:MD-conditions}, and
\eqref{eq:ML-asymptotic}), and we complete the proof showing that
\begin{equation}\label{eq:final-step-for-MD-part1}
\lim_{t\to\infty}A_2(t)=\tilde{\Lambda}(\underline{\theta})
\end{equation}
where $\tilde{\Lambda}$ is the function in
\eqref{eq:def-Lambda-tilde}. Now we consider the Taylor formula
for $\Lambda$, and we have
$$\Lambda(\underline{\eta})=\Lambda(\underline{0})+\langle\nabla\Lambda(\underline{0}),\underline{\eta}\rangle
+\frac{1}{2}\langle\underline{\eta},H_\Lambda(\underline{0})\underline{\eta}\rangle+o(\|\underline{\eta}\|^2)
=\frac{1}{\nu}\cdot(s(\underline{\lambda}))^{1/\nu-1}\langle
\underline{\lambda},\underline{\eta}\rangle+\frac{1}{2}\langle\underline{\eta},C\underline{\eta}\rangle+o(\|\underline{\eta}\|^2),$$
where $\frac{o(\|\underline{\eta}\|^2)}{\|\underline{\eta}\|^2}\to
0$ as $\|\underline{\eta}\|\to 0$ (we have taken in into account
$\Lambda(\underline{0})=0$, \eqref{eq:gradiente-Lambda-0} and the
equality $H_\Lambda(\underline{0})=C$ which can be checked by
inspection); then, after some computations where we take into
account \eqref{eq:def-Lambda-tilde}, we obtain
\begin{multline*}
A_2(t)=ta_t\left(\frac{1}{\sqrt{ta_t}}\frac{1}{\nu}\cdot(s(\underline{\lambda}))^{1/\nu-1}\langle
\underline{\lambda},\underline{\theta}\rangle+\frac{1}{2}\frac{1}{ta_t}\langle\underline{\theta},C\underline{\theta}\rangle+o\left(\frac{1}{ta_t}\right)
-\frac{1}{\sqrt{ta_t}}\frac{E_{\nu,\nu}(s(\underline{\lambda})t^\nu)}{E_{\nu,1}(s(\underline{\lambda})t^\nu)}
\frac{t^{\nu-1}}{\nu}\langle\underline{\theta},\underline{\lambda}\rangle\right)\\
=\tilde{\Lambda}(\underline{\theta})+ta_to\left(\frac{1}{ta_t}\right)
+\frac{\langle\underline{\theta},\underline{\lambda}\rangle}{\nu}\cdot
\sqrt{ta_t}\left((s(\underline{\lambda}))^{1/\nu-1}-\frac{E_{\nu,\nu}(s(\underline{\lambda})t^\nu)}{E_{\nu,1}(s(\underline{\lambda})t^\nu)}t^{\nu-1}\right).
\end{multline*}
Then we get \eqref{eq:final-step-for-MD-part1} if we prove that
$$\lim_{t\to\infty}\sqrt{ta_t}\left((s(\underline{\lambda}))^{1/\nu-1}-\frac{E_{\nu,\nu}(s(\underline{\lambda})t^\nu)}
{E_{\nu,1}(s(\underline{\lambda})t^\nu)}t^{\nu-1}\right)=0.$$ This
is true because
$\left((s(\underline{\lambda}))^{1/\nu-1}-\frac{E_{\nu,\nu}(s(\underline{\lambda})t^\nu)}{E_{\nu,1}(s(\underline{\lambda})t^\nu)}t^{\nu-1}\right)$
goes to zero exponentially fast (as $t\to\infty$), and therefore
it goes to zero faster than the possible divergence of
$\sqrt{ta_t}$; in fact, for two suitable remainder terms $r_1(t)$
and $r_2(t)$ concerning the expansions of Mittag-Leffler functions
in \eqref{eq:ML-asymptotic-with-remainder}, we have
\begin{multline*}
(s(\underline{\lambda}))^{1/\nu-1}-\frac{E_{\nu,\nu}(s(\underline{\lambda})t^\nu)}{E_{\nu,1}(s(\underline{\lambda})t^\nu)}t^{\nu-1}
=\frac{(s(\underline{\lambda}))^{1/\nu-1}E_{\nu,1}(s(\underline{\lambda})t^\nu)-E_{\nu,\nu}(s(\underline{\lambda})t^\nu)t^{\nu-1}}
{E_{\nu,1}(s(\underline{\lambda})t^\nu)}\\
=\frac{(s(\underline{\lambda}))^{1/\nu-1}\left(\frac{1}{\nu}e^{(s(\underline{\lambda}))^{1/\nu}\cdot
t}+r_1(t)\right)-\left(\frac{1}{\nu}(s(\underline{\lambda})t^\nu)^{1/\nu-1}e^{(s(\underline{\lambda}))^{1/\nu}\cdot
t}+r_2(t)\right)t^{\nu-1}}{\frac{1}{\nu}e^{(s(\underline{\lambda}))^{1/\nu}\cdot
t}+r_1(t)}\\
=\frac{(s(\underline{\lambda}))^{1/\nu-1}r_1(t)-r_2(t)t^{\nu-1}}{\frac{1}{\nu}e^{(s(\underline{\lambda}))^{1/\nu}\cdot
t}+r_1(t)}.
\end{multline*}
Thus \eqref{eq:final-step-for-MD-part1} holds, and the proof of
the proposition is complete.
\end{proof}

The following remarks concern Proposition \ref{prop:MD}. In
particular Remark \ref{rem:MD-nu=1} has some connections with
Remark \ref{rem:LD-nu=1} presented above.

\begin{remark}[The rate function $\tilde{\Lambda}^*$ when $C$ is invertible]\label{rem:invertible-C}
If $C$ is invertible one can check that, for all
$\underline{x}\in\mathbb{R}^m$,
$$\tilde{\Lambda}^*(\underline{x}):=\langle C^{-1}\underline{x},\underline{x}\rangle-\tilde{\Lambda}(C^{-1}\underline{x})
=\frac{1}{2}\langle\underline{x},C^{-1}\underline{x}\rangle.$$
\end{remark}

\begin{remark}[The case $\nu=1$]\label{rem:MD-nu=1}
We have
$$c_{jk}^{(1)}:=\left\{\begin{array}{ll}
0&\ \mbox{if}\ j\neq k\\
\lambda_j&\ \mbox{if}\ j=k
\end{array}\right.$$
by \eqref{eq:def-covariance-matrix}. Moreover $C$ is invertible
(since $\underline{\lambda}\in(0,\infty)^m$) and, by Remark
\ref{rem:invertible-C}, we have
$$\tilde{\Lambda}_{(\nu=1)}^*(\underline{x}):=\left\{\begin{array}{ll}
\frac{1}{2}\sum_{i=1}^m\frac{x_i^2}{\lambda_i}&\ \mbox{if}\ \underline{x}\in[0,\infty)^m\\
\infty&\ \mbox{otherwise}.
\end{array}\right.$$
We can also say that
$\tilde{\Lambda}_{(\nu=1)}^*(\underline{x})=\sum_{i=1}^m\tilde{I}_{1,\lambda}^{(A)}(x_i)$
for all $\underline{x}\in\mathbb{R}^m$, where
$\tilde{I}_{\nu,\lambda}^{(A)}$ is the rate function
$\tilde{\Lambda}_{(\nu=1)}^*$ for $m=1$. This agrees with what we
said in Remark \ref{rem:LD-nu=1} (in particular we mean the
independence of the one-dimensional marginal processes
$\{M_1^1(t):t\geq 0\},\ldots,\{M_m^1(t):t\geq 0\}$).
\end{remark}

\begin{remark}[Asymptotic Normality]\label{rem:Asymptotic-Normality}
The computations in the proof of Proposition \ref{prop:MD} still
work even if $a_t=1$ (a case in which the first condition in
\eqref{eq:MD-conditions} fails). Then
$\frac{\underline{M}^\nu(t)-\mathbb{E}[\underline{M}^\nu(t)]}{\sqrt{t}}$
converges weakly (as $t\to\infty$) to the centered Normal
distribution with covariance matrix $C$.
\end{remark}

\section{Statistical applications}\label{sec:statistical-applications}
In this section we present an estimator $\hat{\mathcal{V}}_t$ of
$\nu$, and the vector $\underline{\lambda}$ is assumed to be
known. The aim is to present some asymptotic results (as
$t\to\infty$).

In particular we also assume that $s(\underline{\lambda})\geq 1$.
In fact the function $f_a:(0,\infty)\to(0,\infty)$ defined by
$f_a(x):=\frac{1}{x}\cdot a^{1/x}$ is invertible if $a\geq 1$
(this can be checked noting that
$$f_a^\prime(x)=\frac{a^{1/x}(-\frac{1}{x}\log a-1)}{x^2},$$
and therefore $f_a^\prime(x)<0$ on $(0,\infty)$); then, since
$s(\underline{\lambda})\geq 1$, we consider the estimator defined
by
\begin{equation}\label{eq:estimator-definition}
\hat{\mathcal{V}}_t:=g_{s(\underline{\lambda})}\left(\frac{s(\underline{M}^\nu(t))}{t}\right),
\end{equation}
where $g_{s(\underline{\lambda})}$ is the inverse of
$f_{s(\underline{\lambda})}$. It is quite natural to consider this
estimator because of its consistency; in fact
$\frac{s(\underline{M}^\nu(t))}{t}$ converges to
$\frac{1}{\nu}\cdot(s(\underline{\lambda}))^{1/\nu}$ (as
$t\to\infty$), which is the sum of the components of the vector in
\eqref{eq:gradiente-Lambda-0}.

It is also worth noting that the argument of
$g_{s(\underline{\lambda})}$ can be equal to zero; so we need to
consider
$f_{s(\underline{\lambda})},g_{s(\underline{\lambda})}:[0,\infty]\to[0,\infty]$
where
$f_{s(\underline{\lambda})}(0)=g_{s(\underline{\lambda})}(0)=\infty$,
$f_{s(\underline{\lambda})}(0)=g_{s(\underline{\lambda})}(0)=\infty$
and $[0,\infty]$ is endowed with a suitable topology (an extended
version of the one on $(0,\infty)$) with respect to which
$f_{s(\underline{\lambda})},g_{s(\underline{\lambda})}:[0,\infty]\to[0,\infty]$
are continuous functions between Hausdorff topological spaces. The
continuity of $g_{s(\underline{\lambda})}$ is required for the
application of the contraction principle (see e.g. Theorem 4.2.1
in \cite{DemboZeitouni}) in the proof of the next proposition.

\begin{proposition}\label{prop:LD-estimators}
Assume that $s(\underline{\lambda})\geq 1$. Then the family of
random variables $\left\{\hat{\mathcal{V}}_t:t>0\right\}$
satisfies the LDP with speed $t$ and good rate function $J_\nu$
defined by
$$J_\nu(\hat{\nu}):=\left\{\begin{array}{ll}
\frac{\nu}{\hat{\nu}}\cdot(s(\underline{\lambda}))^{1/\hat{\nu}}\log\left(\frac{\nu}{\hat{\nu}}\cdot(s(\underline{\lambda}))^{1/\hat{\nu}-1/\nu}\right)
-\frac{\nu}{\hat{\nu}}\cdot(s(\underline{\lambda}))^{1/\hat{\nu}}+(s(\underline{\lambda}))^{1/\nu}&\ \mbox{if}\ \hat{\nu}\geq 0\\
\infty&\ \mbox{if}\ \hat{\nu}<0.
\end{array}\right.$$
\end{proposition}
\begin{proof}
If we combine Proposition \ref{prop:LD} and the contraction
principle, the desired LDP holds with speed $t$ and good rate
function $J_\nu$ defined by
$$J_\nu(\hat{\nu}):=\inf\{\Lambda^*(\underline{x}):g_{s(\underline{\lambda})}(s(\underline{x}))=\hat{\nu}\}.$$
So in what follows we manipulate the expression of $J_\nu$ here to
meet its expression in the statement of the proposition. The case
$\hat{\nu}<0$ is trivial because we have the infimum over the
empty set; thus, from now on, we restrict the attention on the
case $\hat{\nu}\geq 0$. Firstly we take into account the
expression of $\Lambda^*$ in Remark
\ref{rem:LD-alternative-expression-rf}, and we have
\begin{multline*}
J_\nu(\hat{\nu})=\inf\{\Lambda^*(\underline{x}):s(\underline{x})=f_{s(\underline{\lambda})}(\hat{\nu})\}\\
=f_{s(\underline{\lambda})}(\hat{\nu})\inf\left\{H\left(\frac{\underline{x}}{f_{s(\underline{\lambda})}(\hat{\nu})};
\frac{\underline{\lambda}}{s(\underline{\lambda})}\right):s(\underline{x})=f_{s(\underline{\lambda})}(\hat{\nu})\right\}
+I_{\nu,s(\underline{\lambda})}^{(A)}(f_{s(\underline{\lambda})}(\hat{\nu})).
\end{multline*}
Moreover the first term  is equal to zero; in fact, if
$f_{s(\underline{\lambda})}(\hat{\nu})>0$, for
$\underline{y}=\frac{f_{s(\underline{\lambda})}(\hat{\nu})}{s(\underline{\lambda})}\cdot\underline{\lambda}$
we have
$$\inf\left\{H\left(\frac{\underline{x}}{f_{s(\underline{\lambda})}(\hat{\nu})};
\frac{\underline{\lambda}}{s(\underline{\lambda})}\right):s(\underline{x})=f_{s(\underline{\lambda})}(\hat{\nu})\right\}=
H\left(\frac{\underline{y}}{f_{s(\underline{\lambda})}(\hat{\nu})};
\frac{\underline{\lambda}}{s(\underline{\lambda})}\right)=0.$$ In
conclusion we have
\begin{multline*}
J_\nu(\hat{\nu})=I_{\nu,s(\underline{\lambda})}^{(A)}(f_{s(\underline{\lambda})}(\hat{\nu}))=
f_{s(\underline{\lambda})}(\hat{\nu})\log\left(\frac{\nu^\nu(f_{s(\underline{\lambda})}(\hat{\nu}))^\nu}{s(\underline{\lambda})}\right)-\nu
f_{s(\underline{\lambda})}(\hat{\nu})+(s(\underline{\lambda}))^{1/\nu}\\
=\frac{\nu}{\hat{\nu}}\cdot(s(\underline{\lambda}))^{1/\hat{\nu}}\log\left(\frac{\nu}{\hat{\nu}}\cdot(s(\underline{\lambda}))^{1/\hat{\nu}-1/\nu}\right)
-\frac{\nu}{\hat{\nu}}\cdot(s(\underline{\lambda}))^{1/\hat{\nu}}+(s(\underline{\lambda}))^{1/\nu}
\end{multline*}
and this completes the proof.
\end{proof}

\begin{remark}[On the probability to have a bad estimate]\label{rem:bad-estimate}
The estimator $\hat{\mathcal{V}}_t$ can provide a bad estimate of
$\nu$ when is larger than 1. However we can say that the event
$\{\hat{\mathcal{V}}_t>1\}$ occurs with an exponentially small
probability; in fact we have $\lim_{t\to\infty}\frac{1}{t}\log
P(\{\hat{\mathcal{V}}_t>1\})=-J_\nu(1)$.
\end{remark}

\begin{remark}[An alternative expression of $J_\nu$]\label{rem:LD-alternative-expression-rf-estimators}
Let us consider the function $D(\cdot;\cdot)$ be defined by
$$D(\lambda_1;\lambda_2):=\lambda_1\log\frac{\lambda_1}{\lambda_2}-\lambda_1+\lambda_2$$
for $\lambda_1\geq 0$ and $\lambda_2>0$. Then, for $\hat{\nu}\geq
0$, we have
$$J_\nu(\hat{\nu})=D\left(\frac{\nu}{\hat{\nu}}\cdot(s(\underline{\lambda}))^{1/\hat{\nu}};(s(\underline{\lambda}))^{1/\nu}\right).$$
\end{remark}

The following corollary provides the asymptotic decay of the
probability of first kind error  for the hypothesis testing
$$H_0:\nu=\nu_0\ \mbox{versus}\ H_1:\nu=\nu_1,\ \mbox{with}\ \nu_0\neq\nu_1.$$
More precisely we mean $P_{H_0}(R_k)$ where $R_k$ is the critical
region defined by
$$R_k:=\left\{\begin{array}{ll}
\{\hat{\mathcal{V}}_t\geq k\}&\ \mbox{if}\ \nu_0<\nu_1,\ \mbox{for some}\ k>\nu_0\\
\{\hat{\mathcal{V}}_t\leq k\}&\ \mbox{if}\ \nu_0>\nu_1,\ \mbox{for
some}\ k<\nu_0.
\end{array}\right.$$

\begin{corollary}\label{corollary:hypothesis-testing}
Assume that $s(\underline{\lambda})\geq 1$. Then
$\lim_{t\to\infty}\frac{1}{t}\log P_{H_0}(R_k)=-J_{\nu_0}(k)$.
\end{corollary}
\begin{proof}
We have
$$\lim_{t\to\infty}\frac{1}{t}\log P_{H_0}(R_k)=-\left\{\begin{array}{ll}
\inf\{J_{\nu_0}(\hat{\nu}):\hat{\nu}\geq k\}&\ \mbox{if}\ \nu_0<\nu_1\\
\inf\{J_{\nu_0}(\hat{\nu}):\hat{\nu}\leq k\}&\ \mbox{if}\
\nu_0>\nu_1.
\end{array}\right.$$
Then, by taking into account the allowed range of values for $k$,
the proof is complete if we show that $J_{\nu_0}(\hat{\nu})$ is
decreasing if $\hat{\nu}<\nu_0$ and is increasing if
$\hat{\nu}>\nu_0$ (note that $J_{\nu_0}(\nu_0)=0$). In order to do
that we recall that the monotonicity intervals for $\lambda_1$
(when $\lambda_2$ is fixed) of the function
$D(\lambda_1;\lambda_2)$ in Remark
\ref{rem:LD-alternative-expression-rf-estimators}: it is
decreasing for $\lambda_1\in(0,\lambda_2)$, is increasing for
$\lambda_1\in(\lambda_2,\infty)$, and $D(\lambda_2;\lambda_2)=0$.
Then, since
$f_{s(\underline{\lambda})}(\hat{\nu})=\frac{1}{\hat{\nu}}\cdot(s(\underline{\lambda}))^{1/\hat{\nu}}$
is decreasing,
$J_{\nu_0}(\hat{\nu})=D\left(\frac{\nu_0}{\hat{\nu}}\cdot(s(\underline{\lambda}))^{1/\hat{\nu}};(s(\underline{\lambda}))^{1/\nu_0}\right)$
decreases (to zero) when $\hat{\nu}$ moves from 0 to $\nu_0$, and
increases (from zero) when $\hat{\nu}$ moves from $\nu_0$ to
infinity.
\end{proof}

\end{document}